\documentclass[12pt,reqno]{amsart}

\usepackage{ccfonts}
\usepackage[T1]{fontenc}
\usepackage[cp1251]{inputenc}
\usepackage{amsmath,amsxtra,amssymb,eufrak}
\usepackage[all]{xy}
\usepackage{mathrsfs}
\usepackage{paralist}
\usepackage[active]{srcltx} 

\newtheorem{thm}{Theorem}
\newtheorem{lm}[thm]{Lemma}

\newtheorem{pr}[thm]{Proposition}

\theoremstyle{definition}

\newcommand{\ptn}{\mathbin{\widehat{\otimes}}}
\newcommand{\GL}{\mathop{\mathrm{GL}}\nolimits}
\newcommand{\Rad}{\mathop{\mathrm{Rad}}\nolimits}

\newcommand{\CC}{\mathbb{C}}

\newcommand{\N}{\mathbb{N}}
\newcommand{\R}{\mathbb{R}}

\newcommand{\cO}{\mathcal{O}}
\newcommand*{\cR}{\mathcal R}

\renewcommand*{\Re}{\mathop{\mathrm{Re}}}
\renewcommand{\le}{\leqslant}
\renewcommand{\ge}{\geqslant}

\let \al         =\alpha

\let \te         =\theta

\let \la         =\lambda
\let \si         =\sigma
\let \up         =\upsilon
\let \om         =\omega

\let \phi         =\varphi

\title
{On density of polynomials in the algebra of holomorphic functions of exponential type\\ on a linear Lie group}
\author{O. Yu. Aristov}
\address {Institute for Advanced Study in Mathematics, Harbin Institute of Technology,  Harbin 150001, China}
\email{aristovoyu@inbox.ru}
\begin{document}
\begin{abstract}
It is shown by the author in [J. Lie Theory 29:4, 1045–1070, 2019] that for every connected linear complex Lie group the algebra of polynomials (regular functions) is dense in the algebra of holomorphic functions of exponential type. However, the argument is quite involved. Here we present a short proof.
\end{abstract}

 \maketitle

 \markright{On density of polynomials}

Following \cite[\S\,5.3.1]{Ak08}  we say that a holomorphic function $f$ on a complex Lie group~$G$ is of
\emph{exponential type} if there is a  submultiplicative weight $\om$ such that $|f(g)|\le \om(g)$ for each $g \in G$. (A \emph{submultiplicative weight} is a locally bounded non-negative function such that $\om(gh)\le \om(g)\,\om(h)$ for every $g, h \in G$.) Denote the set of holomorphic functions of exponential type by $\cO_{exp}(G)$.  Endowed with the inductive topology and the point-wise multiplication $\cO_{exp}(G)$ is
a complete locally convex algebra with jointly continuous multiplication  \cite[Lemma 5.2]{ArAnF}.

Recall that a complex Lie group $G$ is said to be \emph{linear} if it admits a finite-dimensional faithful holomorphic representation. Suppose additionally that $G$ is connected. So $G$ is a connected Stein group and hence admits a canonical structure of complex algebraic affine variety \cite[Th\'{e}or\`{e}me 2]{Ma60}. Thus we can consider the algebra $\cR(G)$  of regular functions (polynomials)  on~$G$. It is well known that $G$ is a semidirect product, $B\rtimes L$, where $B$ is simply connected and solvable, and $L$ is connected and linearly complex reductive; see, e.g.,  \cite[p.\,601, Theorem 16.3.7]{HiNe}. Moreover, $L$ admits a unique structure of affine algebraic group \cite[Appendice]{Ma60} and $B$, while it may be non-algebraic, is biholomorphically equivalent to $\CC^n$ for some $n\in\N$. Hence $G$ can be identified with the affine variety $\CC^n\times L$ and $\cR(G)$ with $\cR(B)\otimes\cR(L)$, where $\cR(B)=\cR(\CC^n)$.
Although the decomposition in the form $B\rtimes L$ is not unique (see, e.g., \cite[p.\,603, Example 16.3.12]{HiNe}), it is easy to see that the structure of affine variety is independent of this decomposition and coincides with the  canonical structure given in \cite{Ma60}.

Our aim is to give a short proof of the following result.

\begin{thm}\label{regden}
\cite[Corollary 5.11]{ArAnF}
Let $G$ be a connected linear complex Lie group. Then $\cR(G)$ is contained and dense in $\cO_{exp}(G)$.
\end{thm}

This result is a key step in proving the holomorphic reflexivity of the algebra of holomorphic functions on a connected linear Lie group; see \cite{ArHR} and the discussion on p.~1046 in \cite{ArAnF}. But the proof given in \cite{ArAnF} is based on a complete description of $\cO_{exp}(G)$, which is quite complicated. In particular, it needs the concept of exponential radical and some facts on analysis on nilpotent Lie groups contained in~\cite{ArAMN}. Here we suggest a shorter argument.

We first prove some auxiliary assertions.
\begin{pr}\label{rclra}
A connected linear complex Lie group admits a finite-dimensional faithful holomorphic representation with closed range.
\end{pr}
The variant of this result for real Lie groups and continuous (automatically differentiable) representations is proved in \cite[Theorem~9]{Go50}; see also \cite[Proposition 5]{Dj76} or  \cite[p.\,597, Theorem 16.2.10]{HiNe}.
In the complex case, we need a lemma, which is not straightforward in contrast to the real case.

Note that when $G$ is connected, the commutator subgroup $(G,G)$ is normal and integral; see, e.g., \cite[p.\,444, Proposition 11.2.4]{HiNe}. If, in addition, $G$ is linear, then $(G,G)$ is closed \cite[Proposition 4.37]{Le02}. So $G/(G,G)$ is a complex Lie group.
\begin{lm}\label{qcomm}
Let $G$ be a connected linear complex Lie group. Then $G/(G,G)$ admits a finite-dimensional faithful holomorphic representation with closed range.
\end{lm}
\begin{proof}
By \cite[Theorem 4.38(iii)]{Le02}, the group $G/\Rad(G,G)$ is linear and hence a Stein group. Then \cite[Lemme~8]{Ma60} implies that $G/(G,G)$ is also a Stein group. Being an abelian Stein group, $G/(G,G)$ is a product of finitely many copies of $\CC$ and $\CC^\times$ (the group of units) \cite[p.\,141, Proposition~4]{MM60}; see also \cite[p.\,190, Theorem XIII.5.9]{Nee}.  Note  that the representation of $\CC$  given by $z\mapsto\begin{pmatrix}
 1& z\\
  0 & 1
\end{pmatrix}$  is faithful and holomorphic, and has closed range. Also, the tautological representation of $\CC^\times$ evidently has closed range.  Therefore $G/(G,G)$ also admits a finite-dimensional faithful holomorphic representation with closed range.
\end{proof}
For technical reasons, in what follows we sometimes use homomorphisms to the general linear group $\GL(\CC, m)$ instead of $m$-dimensional representations.
\begin{proof}[Proof of Proposition~\ref{rclra}]
We argue as in the proof of \cite[Proposition 5]{Dj76} but with the use of Lemma~\ref{qcomm}. The idea is to take two representations, the first being faithful and the second having closed range, with additional assumptions to ensure that their sum satisfies both properties.

Let $\pi\!:\CC\to \GL(\CC, m)$ be a faithful holomorphic homomorphism. By Lemma~\ref{qcomm}, there is a faithful holomorphic homomorphism $G/(G,G)\to \GL(\CC, n)$ with closed range. Denote the composition of  the quotient map $G\to G/(G,G)$ with this homomorphism by $\si$. Put
$$
\rho\!:G\to\GL(\CC, m)\times\GL(\CC, n),\,g\mapsto (\pi(g),\si(g)).
$$
It is clear that $\rho$ is faithful and holomorphic. To show that the range of $\rho$ is closed we apply  \cite[Theorem~1]{Dj76}. (Instead of the last result, which concerns real Lie groups,  a generalization concerning general locally compact groups \cite[Theorem]{Wu90} can be applied.) Put $H\!:=\GL(\CC, m)\times \{1\}$. Then by \cite[Theorem~1]{Dj76}, it suffices to show that  $\rho(G)H$  and  $\rho(G)\cap H$ are closed, and that $\rho(G)$ normalizes $H$.

It is clear that $\rho(G)H=\GL(\CC, m)\times \si(G)$. This set is closed because $\si(G)$ is closed.
Also, $\rho(G)\cap H=\pi((G,G))\times \{1\}$. Treating $\GL(\CC, m)$ as a subgroup of $\GL(\R,2m)$, we can apply \cite[Proposition~2]{Dj76}, which then implies that $\pi((G,G))$ is closed in $\GL(\CC, m)$. So $\rho(G)\cap H$ is closed.
The fact that $\rho(G)$ normalizes $H$ is trivial. Thus $\rho(G)$ is closed.
\end{proof}

Let $\|\cdot\|$ denote the Hilbert space operator norm on the algebra of complex $m\times m$ matrices. Put
\begin{equation}\label{omdef}
\om(a)\!:=\max\{\|a\|,\, \|a^{-1}\|\}\qquad (a\in \GL(\CC, m)).
\end{equation}

\begin{lm}\label{boundC}
Let $m\in\N$ and $\pi\!:\CC\to \GL(\CC, m)$ be a faithful holomorphic homomorphism with closed range.  Then there is $C>0$ such that $|z|\le C\,\om(\pi(z))$ for every $z\in\CC$.
\end{lm}

\begin{proof}
Since $\pi$ is holomorphic, it is not hard to see that there is a generator, i.e., a matrix~$a$ such that $\pi(z)=\exp(za)$ for every $z\in\CC$. We can assume that $a$ is in a Jordan normal form. The following three cases may occur.

(1)~Suppose that $a$ is diagonal and all the eigenvalues are collinear over $\R$. This means that there is $\la\in\CC$ such that the eigenvalues are  $t_1\la,\ldots,t_m\la$
for some $t_1,\ldots,t_m\in\R$.

Let $\te\!:\R\to \CC\!:x\mapsto ix\bar{\la}$. Since the eigenvalues of $\exp(ix\bar{\la}a)$ have modulus $1$ for every $x\in\R$, we can treat $\pi\circ\te$ as a continuous homomorphism from $\R$ to $\mathbb{T}^m$. Since the range of $\pi$ is closed and $\te$ is topologically injective, the range of $\pi\circ\te$ is also closed. Moreover, it is compact because $\mathbb{T}^m$ is compact. So we have a continuous bijective homomorphism from a $\si$-compact group to a compact group.
It is easy to see from the Baire category theorem that every such homomorphism is a topological isomorphism; see, e.g., \cite[Lemma]{Wu90}. Since $\R$ is not compact, we get a contradiction and thus this case is impossible.

(2)~Suppose that $a$ is diagonal and there are two eigenvalues $\la_1$ and $\la_2$ that are not collinear over $\R$.
Then
$$
\om(\pi(z))\ge \exp (\max\{|\Re (\la_1 z)|,\, |\Re (\la_2 z)|\})\ge \max\{|\Re (\la_1 z)|,\, |\Re (\la_2 z)|\}
$$
for every $z\in \CC$. Since $\la_1$ and $\la_2$ that are not collinear, the formula on the right-hand side defines a norm on $\CC$ as a vector space over $\R$. Since all such norms are equivalent, there is $C>0$ such that $|z|\le C\om(\pi(z))$ for every $z\in\CC$ and this completes the proof of
the assertion of the lemma in this case.

(3)~Suppose that $a$ is not diagonal. Taking a suitable Jordan block, we can assume without loss of generality that  the upper-left corner of $a$ is
$\begin{pmatrix}
 \la& 1 \\
  0 & \la
 \end{pmatrix}$  with some $\la\in\CC$.  Then the upper-left corner of $\exp(za)$ has the form $\exp(z\la)\begin{pmatrix}
 1& z\\
  0 & 1
 \end{pmatrix}$. Considering the upper-right entry we have
$$
 \om(\pi(z))\ge \max\{|z\exp(\la z)|,\, |z\exp(-\la z)|\}= |z|\,\exp|\Re(\la z)|\ge|z|
$$
for every $z\in \CC$. Thus the assertion of  the lemma holds also in this case.
\end{proof}

To prove Theorem~\ref{regden} we need the following notation and terminology. Let $G$ be a complex Lie group and $\up\!:G\to (0,+\infty)$ a locally bounded function. Then we consider the Banach space
\begin{equation} \label{cOom}
\cO_\up(G)\!:=\Bigl\{\text{$f$ is holomorphic on $G$ and } |f|_\up\!:=\sup_{g\in
G}{\up(g)}^{-1}{|f(g)|}<\infty\Bigr\}.
\end{equation}
Put also $\cO_{\up^\infty}(G)\!:=\bigcup_{n\in\N}\cO_{\up^n}(M)$ and endow it with the inductive topology.

We say that a function $\eta_1$ on~$G$ is \emph{dominated} by a function $\eta_2$ if there are $K,\al>0$ such that
$$
\eta_1(g)\le K\,\eta_2(g)^\al\quad\text{for all $g\in G$.}
$$
If, in addition, $\eta_2$ is dominated by $\eta_1$, then $\eta_1$ and $\eta_2$ is said to be \emph{equivalent}.

Recall also that if $G$ is compactly generated, i.e.,  generated by a relatively compact neighbourhood~$U$ of the identity, then
\begin{equation}\label{wordlen}
\eta(g)\!: =\min \{2^n \!: \, g \in U^{n}\}
\end{equation}
(where $U^0=\{1\}$) defines a submultiplicative weight, which is called a \emph{word weight}.
If, in addition, $U^{-1}=U$, then $\eta$ is \emph{symmetric}, i.e., $\eta(g^{-1})=\eta(g)$ for every~$g$.

\begin{proof}[Proof of Theorem~\ref{regden}]
Fix a decomposition $G=B\rtimes L$ as above.
Let $\eta$ and $\eta_1$  be symmetric word weights on $G$ and $L$, respectively. The restriction of $\eta$ to~$B$ is also denoted by~$\eta$. Since $G$ is a semidirect product, it follows from  \cite[Proposition 4.2]{ArAnF} that the functions $(b,l)\to \eta(bl)$ and $(b,l)\to\eta(b)\eta_1(l)$ are equivalent on $B\times L$. Then by  \cite[Proposition 5.5(B)]{ArAnF}, we have $\cO_{\eta^\infty}(G)\cong  \cO_{\eta^\infty}(B)\ptn \cO_{\eta_1^\infty}(L)$, where $\ptn$ denotes the complete projective tensor product. Moreover, since $\eta$ and $\eta_1$  are symmetric word weights, it follows from  \cite[Theorem 5.3]{Ak08} that $\cO_{exp}(G)=\cO_{\eta^\infty}(G)$  and $\cO_{exp}(L)=\cO_{\eta_1^\infty}(L)$ as locally convex algebras. Therefore $\cO_{exp}(G)\cong  \cO_{\eta^\infty}(B)\ptn \cO_{exp}(L)$.

Since  $L$ is connected and linearly complex reductive, $\cO_{exp}(L) = \cR(L)$  \cite[Theorem 5.9]{ArAnF}. Since $\cO_{\eta^\infty}(B)\otimes \cO_{exp}(L)$ is dense in $\cO_{\eta^\infty}(B)\ptn \cO_{exp}(L)$ and $\cR(G)=\cR(B)\otimes\cR(L)$, to complete the proof it suffices to show that $\cR(B)$ is contained and dense in $\cO_{\eta^\infty}(B)$.

Proposition~\ref{rclra} implies that there are  $m\in\N$ and a faithful holomorphic homomorphism $\pi\!:G\to\GL(\CC, m)$ with closed range. Since $B$ is simply connected and solvable, it can be represented as an iterated semidirect product  of subgroups $F_1,\ldots, F_n$ each of which is isomorphic to $\CC$. For every $j$ fix an isomorphism $\CC\to F_j$ and denote by $\rho_j$ its  composition with the embedding $F_j\to G$. Then $\pi\circ\rho_j\!:\CC\to\GL(\CC, m)$ is also a faithful holomorphic homomorphism with closed range.

Define $\om$ as in~\eqref{omdef}.  It is well known that every submultiplicative weight is dominated by every word weight; see, e.g., \cite[Theorem 5.3]{Ak08}. Since $\om\circ\pi$ is a submultiplicative weight on $G$, this means
that there are $K>0$ and $\al>0$ such that $\om(\pi(g))\le K\eta(g)^\al$ for all $g\in G$.
Applying Lemma~\ref{boundC}, we conclude that there is $K'>0$  such that
\begin{equation}\label{esteta}
|z| \le K'\eta(\rho_j(z))^\al\qquad \text{for all $j=1,\ldots,n$ and $z\in\CC$}.
\end{equation}

Since $G$ is biholomorphically equivalent to $ F_1\times \cdots F_n\times L$, every element $g$ of $G$ can be identified with $(z_1,\ldots ,z_n,l)$, where $z_j\in \CC$ and  $l\in L$. It follows from \eqref{esteta} that  the function $g\mapsto z_j$ is in $\cO_{\eta^\infty}(B)$ for every $j$ and so the whole $\cR(B)$  (which can be identified with $\CC[z_1,\ldots,z_n]$) is contained in $\cO_{\eta^\infty}(B)$.

Finally, note that every  $f\in\cO_{\eta^\infty}(B)$ is an entire function in the variables $z_1,\ldots,z_n$. Using the bound in Lemma~\ref{boundC},  we conclude that $f$ can be approximated in the topology of $\cO_{\eta^\infty}(B)$ by partial sums of its Taylor series. Thus $\cR(B)$ is dense in $\cO_{\eta^\infty}(B)$.
\end{proof}

\end{document}